\documentclass[a4paper, 11pt,reqno]{amsart}
\usepackage[top = 1in, bottom = 1in, left=1.5in, right=1.5in]{geometry}

\usepackage[utf8]{inputenc}
\usepackage[activeacute,english]{babel}
\usepackage[T1]{fontenc}
\usepackage{verbatim}
\usepackage{lmodern}

\usepackage{xcolor}
\usepackage{hyperref}
\hypersetup{colorlinks, linkcolor={red!50!black}, citecolor={green!50!black}, urlcolor={blue!50!black}}

\usepackage{microtype}

\usepackage[shortlabels]{enumitem}

\usepackage[linesnumbered,ruled,vlined]{algorithm2e}

\usepackage{amsmath,amssymb,amsfonts}
\usepackage{amsthm}
\usepackage{mathrsfs}

\usepackage{tikz}
\usetikzlibrary{decorations.pathreplacing}

\def\NN{\mathbb{N}}

\newcommand{\Gdisc}{G_{\mathrm{disc}}}

\newcommand{\cB}{\mathcal{B}}
\newcommand{\cP}{\mathcal{P}}
\newcommand{\cK}{\mathcal{K}}

\newcommand{\cA}{\mathcal{A}}

\newcommand{\cJ}{\mathcal{J}}
\newcommand{\EE}{\mathbb{E}}
\newcommand{\Hknp}{H^k(n,p)}

\newcommand{\eps}{\varepsilon}
\newcommand{\Bi}{\mathrm{Bi}}
\newcommand{\Be}{\mathrm{Ber}}
\newcommand{\Tstop}{T_{\mathrm{stop}}}
\newcommand{\pathfinder}{\textnormal{\texttt{Pathfinder}}}
\newcommand{\frayalg}{\textnormal{\texttt{Pathbranch}}}
\newcommand{\Forb}{F}

\newcommand{\Part}{\cP}
\newcommand{\IndPart}[2]{\Part_{#1}[#2]}

\theoremstyle{theorem}
\newtheorem{theorem}{\textbf{Theorem}}
\newtheorem{corollary}[theorem]{\textbf{Corollary}}
\newtheorem{lemma}[theorem]{\textbf{Lemma}}
\newtheorem{claim}[theorem]{\textbf{Claim}}
\newtheorem{question-formatless}{\textbf{Question}}
\newtheorem{proposition}[theorem]{\textbf{Proposition}}
\theoremstyle{definition}
\newtheorem{definition}[theorem]{\textbf{Definition}}

\theoremstyle{remark}

\begin{document}

\title[Paths, cycles and sprinkling]{Paths, cycles and sprinkling in random hypergraphs}

\author[O. Cooley]{Oliver Cooley$^{*}$}

\renewcommand{\thefootnote}{\fnsymbol{footnote}}

\footnotetext[1]{Supported by Austrian Science Fund (FWF): I3747.\\
Institute of Discrete Mathematics, Graz University of Technology, Steyrergasse 30, 8010 Graz, Austria. \texttt{cooley@math.tugraz.at}}

\renewcommand{\thefootnote}{\arabic{footnote}}

\begin{abstract}
We prove a lower bound on the length of the longest $j$-tight cycle in a $k$-uniform binomial random hypergraph
for any $2 \le j \le k-1$. We first prove the existence of a $j$-tight path of the required length.
The standard ``sprinkling'' argument is not enough to show that this path can be closed to a $j$-tight cycle --
we therefore show that the path has many extensions, which is sufficient to allow the sprinkling to close
the cycle.
\end{abstract}

\maketitle

\section{Introduction} \label{sec:intro}

\subsection{Paths and cycles in random graphs}

Over the years there has been a considerable amount of research into the length of the longest
paths and cycles in random graphs.
This goes back to the work of Ajtai, Koml\'os and Szemer\'edi~\cite{AjtaiKomlosSzemeredi81},
who showed that in the Erd\H{o}s-R\'enyi binomial random graph $G(n,p)$, the threshold
$p=1/n$ for the existence of a giant component is also the threshold for a path of linear length.
In the supercritical regime, a standard sprinkling argument shows that 
whp\footnote{\emph{with high probability}, meaning with probability tending to $1$
as $n$ tends to infinity.} the lengths of the longest path
and the longest cycle are asymptotically the same, and therefore whp $G(n,p)$ also contains
a cycle of linear length. This has been strengthened by various researchers, including
{\L}uczak~\cite{Luczak91}, and Kemkes and Wormald~\cite{KemkesWormald13}.

We note, however, that when $p=(1+\eps)/n$ for some small $\eps>0$,
the asymptotic length $L_C$ of the longest cycle is still not known precisely: the best known lower and upper bounds
are approximately $4n/3$ (see~\cite{Luczak91}) and $1.7395n$ (see~\cite{KemkesWormald13}) respectively.
On the other hand,
Anastos and Frieze~\cite{AnastosFrieze21} determined the asymptotic length of the longest cycle
precisely
when $p=c/n$ for some sufficiently large constant $c$.

A similar problem, although one requiring very different techniques, is to determine
the length of the longest \emph{induced} path, which was achieved very recently by
Glock~\cite{Glock21} in the regime when $p=c/n$.

\subsection{Paths and cycles in random hypergraphs}

Given an integer $k\ge 2$, a $k$-uniform hypergraph consists of a set $V$ of vertices and
a set $E\subset \binom{V}{k}$ of edges. (A $2$-uniform hypergraph is simply a graph.)
Among the many possible definitions of paths and cycles in hypergraphs, perhaps the most natural and
well-studied is that of \emph{$j$-tight paths and cycles}, which is in fact a family of definitions for $1\le j \le k-1$.

\begin{definition}
Given integers $1\le j \le k-1$ and a natural number $\ell$,
a \emph{$j$-tight path of length $\ell$} in a $k$-uniform hypergraph consists of a sequence of
distinct vertices
$x_1,\ldots,x_{j+(k-j)\ell}$ and a sequence of edges $e_1,\ldots,e_\ell$
such that $e_i = \{x_{(k-j)(i-1)+1},\ldots,x_{(k-j)(i-1)+k}\}$.

A \emph{$j$-tight cycle of length $\ell$} is similar except that
$x_{i} = x_{(k-j)\ell+i}$ for $1 \le i \le j$ (and otherwise all vertices are distinct).
\end{definition}

In the literature, $1$-tight paths/cycles are often called \emph{loose} paths/cycles,
while $(k-1)$-tight is often abbreviated simply to tight.

Let $\Hknp$ denote the $k$-uniform binomial random hypergraph,
in which each $k$-set of vertices forms an edge with probability $p$ independently.
The analogue of the result of Ajtai, Koml\'os and Szemer\'edi
showing a threshold for the existence of a $j$-tight path of linear length
in $\Hknp$ was proved by the author together with Garbe, Hng, Kang, Sanhueza-Matamala and Zalla~\cite{CGHKSZ20}
for all $k$ and $j$.
In contrast to the graph case, in general the threshold is \emph{not} the same
as the threshold for a giant $j$-tuple component (which was determined in~\cite{CooleyKangPerson18}).

Let $a$ be the unique integer satisfying $1 \le a \le k-j$ and $a \equiv k \mod k-j$,
and let
$p_0 = p_0(n,k,j):= \frac{1}{\binom{k-j}{a}\binom{n}{k-j}}$.
The results of~\cite{CGHKSZ20} show that $p_0$ is a threshold for the existence
of a $j$-tight path of linear length in $\Hknp$. Furthermore,
in the case when $p=(1+\eps)p_0$ for some constant $\eps>0$,
upper and lower bounds on the length of the longest $j$-tight path were proved.
In the case when $j\ge 2$, these bounds are $\Theta(\eps n)$ and differ by a factor of $8$.
In the case when $j=1$, the lower bound was $\Theta(\eps^2 n)$ while the upper bound was $\Theta(\eps n)$.

This upper bound in the case when $j=1$ was subsequently improved by the author, Kang and Zalla~\cite{CKZ21}
and shown to be $\Theta(\eps^2 n)$ in the range when $p=(1+\eps)p_0$ (although the results
of that paper also
cover the range $p=c/n$ for any constant $c>1$).
The strategy used was to prove an upper bound on the length of the longest loose
\emph{cycle} which transfers to an upper bound for loose paths using a standard sprinkling argument,
just as has been often observed for graphs.
Similarly, sprinkling can also be used to extend the lower bound on loose paths from~\cite{CGHKSZ20}
to an asymptotically identical lower bound for loose cycles.

\subsection{Sprinkling in hypergraphs}

This raises an obvious question:
can we also use the sprinkling technique for $j\ge 2$, and obtain a $j$-tight cycle from
a $j$-tight path without significantly decreasing the length? Unfortunately, the naive approach does not work.

To see why
first consider the case $j\le k/2$, when we have $p = \Theta(n^{-(k-j)})$
and a path of length $\Theta(n)$. Now for some $\omega\to \infty$,
sprinkle an extra probability of $p/\omega$. We can identify $n/\omega$ many $j$-sets
from the start and from the end of the path with which we attempt to close to a cycle,
and we need a further $k-2j$ vertices from outside the cycle to complete an edge.
Thus the number of potential edges which would close the cycle is $\Theta((n/\omega)^2 n^{k-2j})$,
and the expected number of suitable edges we find is
$$
\frac{p}{\omega} \cdot\Theta\left(\frac{n^{k-2j+2}}{\omega^2}\right) = \Theta\left(\frac{n^{2-j}}{\omega^3}\right).
$$
This will be clearly enough if $j=1$ and if $\omega$ tends to infinity sufficiently slowly, but for $j\ge 2$
the argument fails.
Indeed, for $j > k/2$, the situation becomes even worse: here we even need more than one edge
in order to be able to close the path to a cycle.

The essential reason why the sprinkling no longer works stems from the interplay between
the $j$-sets and the vertices: a $j$-tight path ``lives'' on vertices, but is extended (or closed to a cycle)
via $j$-sets. The number of $j$-sets within the path is naturally bounded by $\Theta(n)$,
but this is tiny compared to the number of $j$-sets in the world (namely $\binom{n}{j}$).

\subsection{Main result}
The main contribution of this paper is to provide a variant of the sprinkling argument which
does work for $j\ge 2$. In particular, we provide a search algorithm which whp will
construct a long $j$-tight cycle in $\Hknp$. We thus provide a lower bound on the length of the longest
$j$-tight cycle. Along the way, we generalise the lower bound for $j$-tight paths
given in~\cite{CGHKSZ20} to be applicable for a larger range of $p$.

Let $L_C=L_C(n,k,j,p)$ be the random variable denoting the length of the longest
$j$-tight path in $\Hknp$.

\begin{theorem}\label{thm:main}
Let $k,j \in \NN$ satisfy $2 \le j \le k-1$ and let $a$ be the unique integer satisfying $1 \le a \le k-j$
and $a \equiv k \mod k-j$.
Let $p_0 = p_0(n,k,j):= \frac{1}{\binom{k-j}{a}\binom{n}{k-j}}$.

For any $\delta>0$, for any constant $c>1$ and for any
sequence $(c_n)_{n \in \NN}$ satisfying $c_n \to c$ the following is true.
Suppose that $p=c_n p_0$. Then
whp
$$
L_C \ge (1-\delta)\cdot \frac{1-c^{-1/(k-j)}}{k-j} \cdot n.
$$
\end{theorem}

Note that it is trivially true that $L_P\ge L_C-O(1)$, where $L_P$ denotes
the length of the longest $j$-tight path in $\Hknp$. Therefore as a corollary
we also obtain a lower bound on $L_P$ which generalises the one in~\cite{CGHKSZ20}.

\section{Preliminaries}

\subsection{Notation and terminology}
In this section we introduce some notation and terminology, and fix various parameters
for the rest of the paper.

Throughout the paper, let $k,j$ be fixed natural numbers satisfying $2 \le j \le k-1$.
In particular, for the rest of the paper we will usually simply refer to paths and cycles
rather than $j$-tight paths and $j$-tight cycles, since $j$ is understood.

All asymptotics in the paper are as $n\to \infty$, and in particular we will use
the standard Landau notation $o(\cdot),O(\cdot),\Theta(\cdot)$ with respect
to these asymptotics. We consider $k,j$ to be constants, so for example
a bound of $O(n)$ may have a constant that is implicitly dependent on $k$ and $j$.

Let us
further define the following parameters.
Let $a=a(k,j)$ be the unique integer satisfying $1 \le a \le k-j$ and
$$
a \equiv k \mod k-j.
$$
The motivation for this parameter will become clear in Section~\ref{sec:structure}.
Given a natural number $\ell$, let $v_\ell = v_\ell(j,k):= j+\ell(k-j)$
denote the number of vertices in a path of length $\ell$.
When $\ell =\Theta(n)$, we will often
approximate $v_\ell$ simply by $\ell(k-j) = (1+O(1/n))v_\ell$.

Let $p_0 = p_0(n,k,j):= \frac{1}{\binom{k-j}{a}\binom{n}{k-j}}$ denote the threshold for a long tight path.
Given $p = p(n)=c_n p_0$ for some sequence $c_n$ of positive real numbers, let
$$
L_1 = L_1(p) := \frac{1-c^{-1/(k-j)}}{k-j} \cdot n.
$$
Note that the parameters
$n,c,k,j$ are implicit in $p$ and will be clear from the context.
Further, let $L_C = L_C(n,p,k,j)$ denote the length of the longest $j$-tight cycle
in $\Hknp$.

For an integer $m$, we denote $[m]:=\{1,\ldots,m\}$ and $[m]_0 := [m] \cup \{0\}$.
We omit floors and ceilings when this does not significantly affect calculations.

\subsection{The structure of paths}\label{sec:structure}

In graphs, there are only two paths with the same edge set (the second is obtained by
reversing orientation), but depending on the values of $k$ and $j$, there may be
many ways of reordering the vertices of a $j$-tight path within the edges which give a different path with the same edges.
For example, in Figure~\ref{fig:74path}, we may re-order $x_1,x_2,x_3$ arbitrarily.
Even in the middle of the path, we may switch the order of $x_8$ and $x_9$ to give a new path.
Nevertheless, we will identify paths which have the same set of edges, and indeed often identify
a path with its edge set. We similarly identify cycles with their edge sets.

\vspace{0.4cm}
\begin{center}
\begin{figure}[h]

\begin{tikzpicture}[level/.style={},decoration={brace,mirror,amplitude=7},scale=1.3]

\fill  (0,0)  circle (0.06)  ;
\fill  (0.6,0) circle (0.06) ;
\fill  (1.2,0) circle (0.06) ;
\fill (1.8,0) circle (0.06);
\fill (2.4,0) circle (0.06);
\fill  (3,0) circle (0.06) ;
\fill (3.6,0) circle (0.06);
\fill (4.2,0) circle (0.06);
\fill  (4.8,0) circle (0.06) ;
\fill (5.4,0) circle (0.06);
\fill (6,0) circle (0.06);
\fill  (6.6,0) circle (0.06) ;
\fill (7.2,0) circle (0.06);
\fill (7.8,0) circle (0.06);
\fill (8.4,0) circle (0.06);
\fill (9,0) circle (0.06);

\node (1) at (0,-0.8) {$x_1$};
\node (2) at (0.6,-0.8) {$x_2$};
\node (3) at (1.2,-0.8) {$x_3$};
\node (4) at (1.8,-0.8) {$x_4$};
\node (5) at (2.4,-0.8) {$x_5$};
\node (6) at (3,-0.8) {$x_6$};
\node (7) at (3.6,-0.8) {$x_7$};
\node (8) at (4.2,-0.8) {$x_8$};
\node (9) at (4.8,-0.8) {$x_9$};
\node (10) at (5.4,-0.8) {$x_{10}$};
\node (11) at (6,-0.8) {$x_{11}$};
\node (12) at (6.6,-0.8) {$x_{12}$};
\node (13) at (7.2,-0.8) {$x_{13}$};
\node (14) at (7.8,-0.8) {$x_{14}$};
\node (15) at (8.4,-0.8) {$x_{15}$};
\node (16) at (9,-0.8) {$x_{16}$};

\draw (1.8,0) ellipse (58pt and 12pt);
\draw[thick, dotted] (3.6,0) ellipse (58pt and 10pt);
\draw[dashed,rounded corners] (3.4,-0.5) rectangle (7.4,0.5);
\draw (7.2,0) ellipse (58pt and 12pt);

\end{tikzpicture}
\caption{A $4$-tight path of length $4$ in a $7$-uniform hypergraph.\label{fig:74path}}
\end{figure}
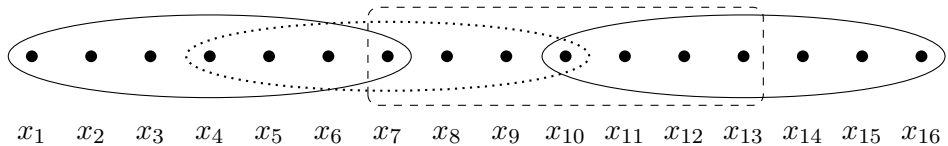

\end{center}

A further important point to note is which $j$-sets we can continue from: For example,
in Figure~\ref{fig:74path}, it seems natural to continue from the $4$-set $\{x_{13},\ldots,x_{16}\}$,
but since the vertices $x_{11},x_{12},x_{13}$ may be rearranged arbitrarily, we could just as well
replace $x_{13}$ by either of $x_{11},x_{12}$ in this $4$-set.

To account for this, we will borrow the following terminology from~\cite{CGHKSZ20}.

\begin{definition}\label{def:extpart}
An \emph{extendable partition} of a $j$-set $J$
is an ordered partition $(C_0, C_1, \dotsc, C_r)$ of $J$,
where $r=\lfloor \frac{j}{k-j} \rfloor$,
with $|C_0| = a$ and $|C_i| = k-j$ for all $i \in [r]$.
\end{definition}

In the example above, the $4$-set $\{x_{10},\ldots,x_{13}\}$ would have
extendable partition $(C_0,C_1)$, where $C_0 = \{x_{10}\}$ and $C_1=\{x_{11},x_{12},x_{13}\}$.
In a search process, the final edge $\{x_{10},\ldots,x_{16}\}$ added to this path
would give rise to three new $4$-sets from which we can continue,
namely $J_i:=\{x_i,x_{14},x_{15},x_{16}\}$, where $i=11,12,13$.
The extendable partition of $J_i$ would be $(C_0^{(i)},C_1^{(i)})$, where $C_0^{(i)} = \{x_i\}$
and $C_1^{(i)} = \{x_{14},x_{15},x_{16}\}$.

For general $k$ and $j$, when we discover an edge $K$ from a $j$-set $J$
with extendable partition $(C_0,\ldots,C_r)$,
the new $j$-sets from which we can continue will be those
consisting of $a$ vertices from $C_1$, all vertices of $C_2,\ldots,C_r$ and all vertices of $K\setminus J$,
and these sets (in this order) will naturally form an extendable partition of the new $j$-set.

We refer the reader to~\cite{CGHKSZ20} for a more detailed discussion of the structure of paths.

\section{Proof outline}

The initial, naive proof idea is to construct a long path using a search process,
and then apply a sprinkling argument to close this path into a cycle. However,
in its most basic form this argument fails for the reasons outlined in the introduction:
we have too few potential attachment $j$-sets and too many required edges for the
sprinkling to work.

Nevertheless, this will still be our overarching strategy, it just needs to be modified slightly.
More precisely, we will aim to construct a \emph{family} of long paths, all of which
are identical along most of their length, but which diverge towards the two ends.
This will give us many more potential attachment $j$-sets, and allow us to push the sprinkling argument through.

As such, we have two main lemmas in the proof.
Let $L_P$ denote the length of the longest path in $\Hknp$.

\begin{lemma}\label{lem:longpath}
Under the conditions of Theorem~\ref{thm:main}, whp $L_P \ge \left(1-\frac{\delta}{3}\right)L_1(p)$.
\end{lemma}

The proof of this lemma is essentially the same as the proof for the special case
of $p=(1+\eps)p_0$ in~\cite{CGHKSZ20}. We first define an appropriate depth-first search process for constructing $j$-tight paths.
Heuristically, this DFS is supercritical for as long as
the path constructed has length significantly smaller than $L_1$.
However, the algorithm will avoid re-using $j$-sets that have already been tried, if they
led to dead-ends, and we need to know that this will not slow down the growth too much.
For this, we need a \emph{bounded degree lemma} which shows that,
in an appropriate sense, these $j$-sets
are evenly distributed in the hypergraph, rather than clustered together.
The proof of Lemma~\ref{lem:longpath} is given in Section~\ref{sec:longpath}.

The main original contribution of this work is the second lemma,
which guarantees the existence of a family of $j$-tight paths with many different endpoints.
The DFS algorithm is well-suited to creating long paths quickly, but in order to fan out
towards the ends, we will switch to a \emph{breadth}-first search algorithm.
The result of this algorithm will be the following structure.

\begin{definition}
Given integers $\ell_1,\ell_2$,
a $j$-set $J$ and a path $P$ of length $\ell_1$ with end $J'$, we say that $J$ \emph{$\ell_2$-augments} the pair $(P,J')$
if there exists a path $P_{J, J'}$ starting at $J$ and ending at $J'$ such that $P_{J, J'} \cup P$ is again a path,
and has length at most $\ell_1+\ell_2$.
\end{definition}

In other words, we can extend $P$ by length at most $\ell_2$ to end at $J$ instead of $J'$.

\begin{lemma}\label{lem:manyends}
Under the conditions of Theorem~\ref{thm:main},
there exists some constant $\eps \in (0,1)$ such that the following holds.
Whp $\Hknp$ contains a $j$-tight path $P_0$ of length $(1-\delta/2)L_1(p)$ with ends $J_s,J_e$,
and collections $\cA,\cB$ of $j$-sets such that:
\begin{itemize}
\item $|\cA|,|\cB| = \eps^2 n^j$;
\item Every $j$-set $A \in \cA$ $2(\log n)^2$-augments $(P_0,J_s)$;
\item Every $j$-set $B \in \cB$ $2(\log n)^2$-augments $(P_0,J_e)$;
\item For at least $(1-\eps)\eps^4 n^{2j}$ pairs $(A,B) \in \cA \times \cB$, the augmenting paths $P_{A,J_s},P_{B,J_e}$ are vertex-disjoint.
\end{itemize}
\end{lemma}

We will show how Lemma~\ref{lem:manyends}
follows from Lemma~\ref{lem:longpath} in Section~\ref{sec:manyends}.
Before continuing with the proofs of these two lemmas,
let us first show how Lemma~\ref{lem:manyends} implies our main theorem.

\begin{proof}[Proof of Theorem~\ref{thm:main}]
Let $\omega$ be some function of $n$ tending to infinity arbitrarily slowly,
and let $p':= (1-1/\omega)p$. We apply Lemma~\ref{lem:manyends} with $p'$ in place of $\omega$.
Let us observe that $p' = c_n' p_0$, where $c_n':= (1-1/\omega)c_n \to c$,
and therefore we have $L_1(p')=L_1(p)$.
It follows that the path $P_0$
provided by Lemma~\ref{lem:manyends} has length at least $(1-\delta/2)L_1(p) \ge (1-\delta)L_1(p)$.

Now for each pair $(A,B)\in \cA \times \cB$ satisfying the last condition of Lemma~\ref{lem:manyends},
concatenating the paths $P_{A,J_s},P_0,P_{B,J_e}$ gives a path $P_{A,B}$ with ends $A$ and $B$
and containing $P_0$, which therefore has length at least $(1-\delta)L_1(p)$ (the length of $P_0$), but also of length at most
$(1-\delta/2)L_1(p') + 2(\log n)^2 = (1-\Theta(1))n$.
In other words, $P_{A,B}$ leaves a set $V_{A,B}$ of $\Theta(n)$ vertices uncovered.

Let us now sprinkle an additional probability of $p'':=p-p'$ onto the hypergraph.
In order to close $P_{A,B}$ to a cycle, we need to find a configuration containing
$s=\lceil \frac{j}{k-j} \rceil$ edges and $b=k-j-a$ vertices of
$V_{A,B}$.
For a fixed choice of $A,B$ and $b$ vertices of $V_{A,B}$,
the probability that the $s$ required edges exist is simply $(p'')^s$.
For fixed $A$ and $B$, but for different choices of the $b$ vertices
of $V_{A,B}$, these edges are all distinct.
However, given two choices $A_1,B_1,R_1$ and $A_2,B_2,R_2$ of $A,B$ and $b$ vertices
from $V_{A,B}$, it is possible that the configurations require the same $k$-set to be an edge,
and thus we no longer have independence. We therefore show that there are sufficiently many
choices for which the $k$-sets \emph{are} all distinct.

To see this, observe that there are $\Theta\left(\eps^4 n^{2j+b}\right)$ choices for
the triple $(A,B,R)$,
and any particular $k$-set is required to be an edge by at most $O(n^{2j+b-k})$ triples.
Therefore any choice of triple shares a $k$-set with at most $O(n^{2j+b-k})$ other triples,
and we may greedily choose $\Theta\left(\eps^4 n^{2j+b-(2j+b-k)}\right) = \Theta(\eps^4 n^k)$
without conflicts.

By choosing this many triples, we observe that the probability that none of them closes a cycle is
$$
(1-(p'')^s)^{\Theta\left(\eps^4 n^k\right)} \le \exp\left(-\Theta\left(\frac{\eps^4n^k}{\omega^s n^{s(k-j)}}\right)\right).
$$
Now recall that $s=\lceil \frac{j}{k-j}\rceil = \lceil\frac{k}{k-j}\rceil -1 \le \frac{k-1}{k-j}$.
Thus the probability that none of the choices of $A,B,R$ admits the edges necessary to close a cycle is at most
$$
\exp\left(-\Theta\left(\frac{\eps^4n}{\omega^{(k-1)/(k-j)}}\right)\right) = o(1),
$$
where the last estimate follows since $\omega$ tends to infinity arbitrarily slowly,
so in particular we have $\omega^{(k-1)/(k-j)} = o(n)$.
\end{proof}

\section{Depth-first search: proof of Lemma~\ref{lem:longpath}}\label{sec:longpath}

Since the proof of Lemma~\ref{lem:longpath} is essentially the same as that of
the special case when $p=(1+\eps)p_0$ from~\cite{CGHKSZ20}, we will not go into
full detail here. However, we will outline the argument, partly to make this paper self-contained
and partly because some of the ideas will reappear in the more complicated
proof of Lemma~\ref{lem:manyends} in Section~\ref{sec:manyends}.

In order to prove the existence of a long path, we borrow the \pathfinder\ algorithm
from~\cite{CGHKSZ20}. This is in essence a depth-first search algorithm; however,
there are a few complications in comparison to the graph case.

Recall from Section~\ref{sec:structure} that, depending on the values of $k$ and $j$,
when we add an edge to the current path, we may have multiple new $j$-sets
from which we could extend the path.
For this reason, each time we increase the length of the path, we produce
a \emph{batch} of $j$-sets with which the path could potentially end. In the example
in Figure~\ref{fig:74path}, the batch would consist of the three $4$-sets containing the
three new vertices and one of the previous three vertices;
more generally, a batch will contain any $j$-set from which the path can be extended if we
discover a further edge containing that $j$-set (and no other vertices from the current path).

During the algorithm, at each time step we will \emph{query} a $k$-set to determine whether it forms an
edge or not. This may be thought of as revealing the outcome of a $\Be(p)$ random variable corresponding
to this $k$-set (with these variables being mutually independent).

We will describe $j$-sets as being
\emph{neutral}, \emph{active} or \emph{explored};
initially all $j$-sets are neutral; a $j$-set $J$ becomes active if we have discovered a
path which can end in $J$ (in which case a whole batch becomes active);
$J$ becomes explored once we have queried all possible $k$-sets from $J$.

Of course, in order to produce a path we will not query any $k$-sets from $J$
that contain any further vertices (apart from $J$) of the current path.
But more than this, in order to allow \emph{analysis} of the algorithm, we place an 
additional restriction: specifically, we do not query any $k$-set that contains any other
active or explored $j$-set. This ensures that we never query the same $k$-set twice
from different $j$-sets, and therefore the outcome of each query is independent of all other queries.

Whenever a new $j$-set becomes active, it is added to the end of the current queue.
Since we are considering a depth-first search, we will always query $k$-sets from
the \emph{last} active $j$-set in the queue.
Whenever the queue of active $j$-sets is empty (so also the current path is empty),
we choose a new neutral $j$-set from which to continue uniformly at random, and this $j$-set becomes active.

A formal description of the \pathfinder\ algorithm can be found in~\cite{CGHKSZ20}.

Let us observe that in the algorithm, whenever we find an edge from a $j$-set with extendable partition $(C_0,\ldots,C_r)$,
$\binom{|C_1|}{a} = \binom{k-j}{a}$ new $j$-sets become active.
Heuristically, towards the start of the process we will query approximately $\binom{n-v_\ell}{k-j}$ many $k$-sets from a $j$-set,
where $\ell$ is the current length of the path
(and recall that $v_\ell = j+\ell(k-j)$ denotes the number of vertices in a path of length $\ell$).
This gives a clear intuition for why we should find a path of length
$L_1(p)$: the expected number of $j$-sets that become active from
any current $j$-set is approximately
$$
\binom{k-j}{a}\binom{n-v_\ell}{k-j} p = \left(1+o(1)\right)\left(1-\frac{\ell(k-j)}{n}\right)^{k-j} c.
$$
When $\ell=L_1 = \frac{1-c^{-1/(k-j)}}{k-j} \cdot n$, up to the $1+o(1)$ error term this gives precisely~$1$---in other words, $L_1$
is the length at which this process changes from being supercritical to subcritical.

The main difficulty in the proof comes in the approximation of the number of $k$-sets that we query
from each $j$-set, which above we estimated by $\binom{n-v_\ell}{k-j}$.
In fact, this is an obvious \emph{upper} bound, whereas we need a \emph{lower} bound.
The upper bound takes account of $k$-sets that may not be queried because they contain a vertex
from the current path, but $k$-sets may also be forbidden because they contain another active or explored $j$-set
(apart from the one we are currently querying from).

We call a $j$-set \emph{discovered} if it is either active or explored. The set $\Gdisc = \Gdisc(t)$ of discovered
$j$-sets at time $t$ may be thought of as the edge set of a $j$-uniform hypergraph. It is intuitive that at the
start of the search process (i.e.\ for small $t$), this hypergraph is sparse, but we need to quantify this more precisely.
Given $0 \le i \le j-1$, let $\Delta_i(t)=\Delta_i(\Gdisc(t))$ denote the maximum $i$-degree of $\Gdisc(t)$,
that is the maximum over all $i$-sets $I$ of the number of $j$-sets of $\Gdisc(t)$ that contain $I$.
(Note in particular that $\Delta_0(t) = |\Gdisc(t)|$.)
The purpose of this parameter is highlighted in the following proposition.
\begin{proposition}\label{prop:dfs:manyqueriesfromJ}
Suppose that a $j$-set $J$ becomes active when the length of the path is $\ell=\ell_J$
and that $n-v_{\ell_J} = \Theta(n)$.
Then the number of $k$-sets that are eligible to be queried from $J$ at time $t$ is at least
$$
\left(1-\sum_{i=0}^{j-1}O\left(\frac{\Delta_i(t)}{n^{j-i}}\right)\right)\binom{n-v_{\ell_J}}{k-j}.
$$
\end{proposition}

\begin{proof}
Let us consider how many $k$-sets may not be queried from a $j$-set $J$
because they contain a second, already discovered $j$-set $J'$. We will make a case distinction based on the possible
intersection size $i= |J \cap J'| \in [j-1]_0$, and note that for each $i\in [j-1]_0$, the number of discovered $j$-sets $J'$ which
intersect $J$ in $i$ vertices is at most $\binom{j}{i} \Delta_i(t)$, and the number of $k$-sets that are forbidden
because they contain both $J$ and $J'$ is (crudely) at most $n^{k-2j+i}$.
Therefore the number of forbidden $k$-sets is certainly at most
\begin{align*}
\sum_{i=0}^{j-1} \binom{j}{i} \Delta_i(t) n^{k-2j+i}
& = \sum_{i=0}^{j-1} O\left(\frac{\Delta_i(t)}{n^{j-i}}\right) \binom{n-v_{\ell_J}}{k-j},
\end{align*}
where the approximation follows because $\binom{n-v_{\ell_J}}{k-j} = \Theta(n^{k-j})$.
\end{proof}

It follows from this proposition that
if $\Delta_i(t) \ll n^{j-i}$ for each $i$, the number of forbidden $k$-sets is insignificant
compared to the number of $k$-sets that may be queried, and the calculation above will go through with
the addition of some smaller order error terms.

We will therefore run the \pathfinder\ algorithm until one of the three stopping conditions is satisfied.
Let us fix a constant $0<\eps \ll \delta$ and further constants $1 \ll c_0 \ll c_1 \ll \ldots \ll c_{j-1} \ll 1/\sqrt{\eps}$.

\begin{enumerate}[(DFS1)]
\item \label{dfsstop:length} $\ell = (1-\delta/3)L_1$;
\item \label{dfsstop:time} $t=\eps^2 n^k =:t_0$;
\item \label{dfsstop:degree} $\Delta_i(t) \ge \eps c_i n^{j-i}$ for some $0 \le i \le j-1$.
\end{enumerate}

Now our goal is simply to show that whp the algorithm terminates when~\ref{dfsstop:length} is invoked.
As such, we have two main auxiliary results.

\begin{proposition}\label{prop:dfsstop:time}
Whp~\ref{dfsstop:time} is not invoked.
\end{proposition}

\begin{lemma}\label{lem:dfsstop:degree}
Whp~\ref{dfsstop:degree} is not invoked.
\end{lemma}

We note that Lemma~\ref{lem:dfsstop:degree} is a form of \emph{bounded degree lemma}
similar to the one first proved in~\cite{CooleyKangPerson18} and subsequently used in
one form or another in~\cite{CGHKSZ20,CooleyKangKoch16,CooleyKangKoch19}.
A far stronger form also appeared in~\cite{CooleyKangKoch18}.
In its original form, the bounded degree lemma roughly states that no $i$-degree
is larger than the average $i$-degree by more than a bounded factor.
The stronger form in~\cite{CooleyKangKoch18} even provides a lower bound, showing
that whp all $i$-degrees are approximately equal, a phenomenon we call \emph{smoothness}.

For our purposes we need only the upper bound, and allow a deviation from the average of $\Theta(c_i/\eps)$.
This could certainly be improved, and it seems likely that even smoothness is satisfied, but since
we do not require an especially strong result for this paper,
for simplicity we make no effort to optimise the parameters.

\begin{proof}[Proof of Proposition~\ref{prop:dfsstop:time}]
Let us suppose (for a contradiction) that at time $t_0= \eps^2 n^k$, neither~\ref{dfsstop:length} nor~\ref{dfsstop:degree} has been invoked.
Since~\ref{dfsstop:degree} has not been invoked, we have $\Delta_i \le \eps c_i n^{j-i} \le \sqrt{\eps} n^{j-i}$ for each $i \in [j-1]_0$,
and so by Proposition~\ref{prop:dfs:manyqueriesfromJ},
from each explored $j$-set we certainly made at least
$$
(1-O(\sqrt{\eps}))\binom{n-v_{(1-\delta/3)L_1}}{k-j} \ge (1+\delta^2) \binom{n-v_{L_1}}{k-j}
$$
queries. We also observe that at time $t_0$ the number of edges we have discovered is distributed as
$\Bi(t_0,p)$, which has expectation $t_0p = \Theta(\eps^2 n^j)$. By a Chernoff bound, whp we have discovered
at least $(1-\delta^3)t_0p$ edges, and therefore at least $(1-\delta^3)t_0p \binom{k-j}{a}$ many $j$-sets have become active.
At any time, the number of currently active $j$-sets is $O(L_1) = O(n)$, and therefore the number of fully explored $j$-sets
is at least 
$$
(1-\delta^3)t_0p \binom{k-j}{a} - O(n) \ge (1-\delta^2/2)t_0p\binom{k-j}{a},
$$
since $t_0p = \Theta(\eps^2 n^j) \gg n$.

Thus the total number of queries made by time $t_0$ is at least
$$
(1-\delta^2/2)t_0p\binom{k-j}{a} (1+\delta^2) \binom{n-v_{L_1}}{k-j}
\ge (1+\delta^2/3)t_0,
$$
which is clearly a contradiction since by definition we have made precisely $t_0$ queries.
\end{proof}

\begin{proof}[Proof outline of Lemma~\ref{lem:dfsstop:degree}]
We give only an outline of the proof here to introduce the main ideas.
An essentially identical argument was used to prove~\cite[Lemma~34]{CGHKSZ20}.

First consider the case
$i=0$, when the desired bound follows from the fact that, by a Chernoff bound we have found at most
$2pt_0 = O(\eps^2 n^j)$ edges, each of which leads to $O(1)$ many $j$-sets becoming active.
Some further $j$-sets may also become active without finding an edge each time the queue of active $j$-sets is empty and we pick 
a new $j$-set from which to start. It is easy to bound the number of times this happens
by $O(t_0 n^{j-k}) = \Theta(\eps^2 n^j)$ (see the argument for ``new starts'' below).

Now given $i \in [j-1]$ and an $i$-set $I$, there are three ways in which
the degree of $i$ in $\Gdisc$ could increase.
\begin{itemize}
\item A \emph{new start} at $I$ occurs when the current path is fully explored and we pick a new (ordered) $j$-set from which
to start a new exploration process. If this $j$-set contains $I$, then the degree of $I$ increases by $1$.
\item A \emph{jump} to $I$ occurs when a $k$-set containing $I$ is queried from a $j$-set not containing $I$ and this $k$-set is indeed an edge.
Then the degree of $I$ increases by at most $\binom{k-j}{a}$.
\item A \emph{pivot} at $I$ occurs when an edge is discovered from a $j$-set already containing $I$.
Then the degree of $I$ increases by at most $\binom{k-j}{a}$.
\end{itemize}

We bound the contributions to the degree of $I$ made by these three possibilities separately.

\subsubsection*{New starts} We can crudely bound the number of new starts by observing that
for each \emph{starting} $j$-set we must certainly have made at least $\Theta(n^{k-j})$ queries
to fully explore it,
and therefore at time $t$ we can have made at most $\Theta(t n^{j-k})$ many new starts in total
(when $t \ge n^{k-j}$).
Since we chose the $j$-set for a new start uniformly at random, the probability that such a new start
contains $I$ is $\Theta(n^{-i})$, and the probability that the number of new starts at $I$
by time $t_0=\eps^2 n^k$
is larger than twice its expectation (which itself is $\Theta (t_0 n^{j-k-i})=\Theta(\eps n^{j-i}) \ge \sqrt{n}$) is exponentially small.

\subsubsection*{Jumps}

We further subclassify jumps according to the size $z$ of the intersection $I\cap J$
between $I$ and the $j$-set $J$ from which the jump to $I$ occurs.
Observe that since for a jump we cannot have $I \subset J$, we must have $0 \le z \le i-1$.
The number of $j$-sets of $\Gdisc(t)$ with intersection of size $z$ with $I$ is at most $\binom{i}{z}\Delta_z(t) = O(\Delta_z(t)) = O(\eps c_z n^{j-z})$,
where for the last approximation we used condition~\ref{dfsstop:degree} with $z$ in place of $i$.
For each such $j$-set $J$, the number of $k$-sets which contain both $J$ and $I$, and which could therefore result in a jump to $I$,
is at most $\binom{n}{k-j-i+z} = O(n^{k-j-i+z})$.

Thus the total number of queries made which could result in jumps to $z$ is at most
$\sum_{z=0}^i O(\eps c_z n^{j-z}) \cdot O(n^{k-j-i+z}) = O(\eps c_{i-1} n^{k-i})$.
Each such query gives a jump with probability $p=O(n^{k-j})$,
and a Chernoff bound implies that the number of jumps is $O(\eps c_{i-1} n^{j-i})$.
Since each jump contributes at most $\binom{k-j}{a}=O(1)$
to the degree of $I$, the total contribution made by jumps is $O(\eps c_{i-1} n^{j-i})$.

\subsubsection*{Pivots}

We observe that from any $j$-set containing $I$, the expected number of pivots at $I$ is at most
$\binom{n}{k-j}p = O(1)$. Furthermore since we are studying a DFS process creating a path,
the number of consecutive pivots at $I$ can be at most $\frac{k-i}{k-j} \le k$ before
the path has left $I$. Since the number of new starts and jumps to $I$ is $O(\eps c_{i-1} n^{j-i})$,
it follows that also whp the number of pivots at $I$ is $O(\eps c_{i-1} n^{j-i})$.

\vspace{0.5cm}

Now we have bounded the contribution to the degree of $I$ made by each of the three possibilities
as $O(\eps c_{i-1} n^{j-i})$, and summing these three terms, together with the fact that $c_{i-1} \ll c_i$, gives the desired result.
\end{proof}

\section{Breadth-first search: Proof of Lemma~\ref{lem:manyends}}\label{sec:manyends}

In this section we aim to show how we can use the single long path
guaranteed whp by Lemma~\ref{lem:longpath} and extend it using a breadth-first search process
to a family of paths with many ends, as required by Lemma~\ref{lem:manyends}.

\subsection{The BFS algorithm: \frayalg}

\subsubsection{Motivation and setup}
We will use Lemma~\ref{lem:longpath} as a black box, and
let $P_0'$ be some path of length $(1-\delta/2)L_1+2(\log n)^2 \le (1-\delta/3)L_1$,
which is guaranteed to exist whp.
Let $P_0$ be the subpath of length $(1-\delta/2)L_1$ obtained by removing
$(\log n)^2$ edges from each end of $P_0'$.
Furthermore, let $\cJ_1,\cJ_2$ be the collections of $(\log n)^2$ many $j$-sets
which are ends of a subpath of $P_0'$, but not of $P_0$, divided naturally into two
collections according to which end of $P_0'$ they are closest to.

Our aim is to start two breadth-first processes starting at $\cJ_1,\cJ_2$ to extend $P_0$,
and to show that these processes quickly grow large. This fact in itself would be easy to prove
by adapting the proof strategy from Section~\ref{sec:longpath},
since the length of $P_0$ is such that the processes are (just) supercritical, and intuitively
we only need a logarithmic number of steps to grow to polynomial size.

More delicate, however,
is to show that the search process produces path ends that are compatible with each other,
in the sense that there are many choices of pairs of ends between which we have a path.
In order to construct compatible sets of ends, having run the algorithm once
to find augmenting paths at one end, we will have a set $\Forb$ of \emph{forbidden vertices};
roughly speaking, these are vertices which lie in too many of the augmenting paths
from the first application of the algorithm,
and therefore we would like to avoid them when constructing augmenting paths
at the other end.

\subsubsection{Informal description}

Let us first describe the algorithm informally. We will start with 
paths $P_0 \subset P_0'$ and a set of ends $\cJ$ (which will be either $\cJ_1$
or $\cJ_2$). These ends come with the natural extendable partition induced by $P_0'$.
As in the DFS algorithm, we will label $j$-sets as \emph{neutral}, \emph{active}
or \emph{explored}. Initially the $j$-sets of $\cJ$ are active and all others are neutral.

At each time $t$ we will \emph{query} a $k$-set containing an active $j$-set $J$ to determine whether it is an edge.
If it is, then $\binom{k-j}{a}$ new $j$-sets are potential ends with which we can extend the
path from $J$, and these become active, also inheriting an appropriate extendable partition.
In order to ensure that we are always creating a path, we will forbid queries of $k$-sets
which contain vertices of the path ending in $J$ (except the vertices of $J$ itself).
We will also forbid $k$-sets with vertices from the forbidden set~$\Forb$. Finally, to ensure independence
of the queries we will forbid $k$-sets which contain some explored $j$-set.
(Note that because we are using a breadth-first search, we do not need to exclude other active
$j$-sets $J'$ because $J'$ will be dealt with later once $J$ is explored, and such $k$-sets will be forbidden from $J'$ because
they contain $J$.)

We will proceed in a standard BFS manner, i.e.\ from the first active $j$-set in the queue
we will query all permissible $k$-sets, and any new $j$-sets we discover are added
to the end of the queue.

We note that during the BFS process, the $j$-sets which become explored
including those in $\cJ_1 \cup \cJ_2$, are certainly ends of a path containing $P_0$,
and therefore candidates in our later sprinkling step. Since we have used Lemma~\ref{lem:longpath}
as a black box, and consider the BFS algorithm as a fresh start, we initially have a blank slate
of explored $j$-sets, and therefore any $j$-set which is explored or active during the new process
is an appropriate end.

\subsubsection{Formal description}

Given a path $P$ and a $j$-set $J$ which is an end of some subpath $P'$ of $P$,
let us denote by $\IndPart{J}{P}$ the extendable partition of $J$ which is naturally
induced by the path $P'$.
We will also denote by $P|_J$ the longer of the two maximal subpaths of $P$ ending in $J$.

The formal description of the \frayalg\ algorithm appears below.

\renewcommand{\thealgocf}{}

\SetAlFnt{\footnotesize}
\begin{algorithm}
	\DontPrintSemicolon
	\KwIn{Integers $k, j$ such that $1\le j \le k-1$.}
	\KwIn{$H$, a $k$-uniform hypergraph.}
	\KwIn{Paths $P_0\subset P_0'$, set of $j$-sets $\cJ \in P_0'\setminus P_0$}
	\KwIn{$\Forb$, a set of forbidden vertices of size at most $\delta^2 n$}
	Let $a \in [k-j]$ be such that $a \equiv k \bmod (k-j)$\;
	Let $r=\lceil \frac{j}{k-j}\rceil-1$\;
	$A \gets \cJ$ ordered lexicographically \tcp*{active $j$-sets}
	$N \gets \binom{V(H)}{j} \setminus(\cJ_1\cup \cJ_2)$ \tcp*{neutral $j$-sets}
	$E \gets \emptyset$ \tcp*{explored $j$-sets}
	\ForAll{$J \in \cJ$}{
	$P_J \gets P_0'|_J$ \tcp*{current $j$-tight path to $J$}
	$\ell_J \gets |P_J|$ \tcp*{length of $P_J$}
	$\cP_J \gets \IndPart{J}{P_0'}$ \tcp*{extendable partition of $J$}
	}
	$t \gets 0$ \tcp*{``time'', number of queries made so far}
	
	\While{$A \neq \emptyset$}{
		Let $J$ be the first $j$-set in $A$\;
		Let $\cK$ be the set of $k$-sets $K \subset V(H)$ such that
			$K \supset J$,
			such that $K\setminus J$ is vertex-disjoint from $P_J$, from $P_0'$ and from $\Forb$,
			and such that $K$ does not contain any $J' \in E$\;
		\While{$\cK \neq \emptyset$}{
			Let $K$ be the first $k$-set in $\cK$ according to the lexicographic order\;
			$t \gets t+1$ \tcp*{a new query is made}
			\If(\tcp*[f]{``query $K$''}){$K \in H$}{
				Let $(C_0, C_1, \dotsc, C_r)$ be the extendable partition of~$J$\;
				\For{each $Z \in \binom{C_1}{a}$}{
					$J_Z \gets Z \cup C_2 \cup \dotsb \cup C_r \cup (K \setminus J)$ \tcp*{$j$-set to be added}
					$P_{J_Z} \gets P_J + K$ \tcp*{Path ending at $J_Z$}
					$\cP_{J_Z} \gets (Z, C_2, \dotsc, C_r, K \setminus J)$\tcp*{extendable partition}
					$\ell_{J_Z} \gets \ell_J + 1$ \;
					$A \gets A + J_Z$ \tcp*{$j$-set becomes active}
					$N \gets N-J_Z$ \tcp*{$j$-set is no longer neutral}
				}
			}
			$(A_t, E_t) \gets (A, E)$ \tcp*{update ``snapshot'' at time $t$}
			$\cK \gets \cK -K$ \tcp*{update $\cK$}
		}
		$E \gets E+J$ \tcp*{$J$ becomes explored}
		$A \gets A-J$ \tcp*{$J$ is no longer active}
		}
	\caption{\frayalg}
	\label{algorithm:dfs}
\end{algorithm}

We will run this algorithm twice, once with $\cJ = \cJ_1$ and once with $\cJ=\cJ_2$.
Of course, the two instances of the algorithm will not be independent of each other in general.
However, if we can show that each instance satisfies some desired properties whp,
a union bound shows that also whp both instances satisfy these properties.
We will subsequently show that the desired properties will be enough to combine the two outputs of the algorithm
in an appropriate way.

\subsection{Analysing the algorithm}

Let us define
$\ell_t := \max_{J\in A_t} \ell_J$.
Let us also fix a constant $0<\eps \ll \delta$ and further constants $1 \ll c_0 \ll c_1 \ll \ldots \ll c_{j-1} \ll 1/\eps$. As in the DFS algorithm,
let $\Gdisc(t):= A_t \cup E_t$, and we denote $\Delta_i(t):=\Delta_i(\Gdisc(t))$
for any $0\le i \le j-1$ and $t \in \NN$.
We will run the \frayalg\ algorithm until time $\Tstop$, the first time at which one of the following stopping conditions is satisfied.

\begin{enumerate}[(S1)]
\item \label{bfsstop:death} The algorithm has terminated.
\item \label{bfsstop:length} $\ell_t = \left(1-\delta/2 \right)L_1 + 2(\log n)^2 =:L_0$.
\item \label{bfsstop:degree} $\Delta_i(t) \ge \eps c_i n^{j-i}$ for some $i$.
\item \label{bfsstop:manyends} $|\Gdisc(t)| = \eps^2 n^j$.
\end{enumerate}

Our principal aim is to show that whp it is~\ref{bfsstop:manyends} that is invoked.
The following proposition will be critical.

\begin{proposition}\label{prop:bfs:manyqueriesfromJ}
For any $t\le \Tstop$, the number of $k$-sets which are eligible to be queried from an active $j$-set
is at least $\binom{n-v_{(1-\delta/3)L_1}}{k-j} = (1+\Theta(\delta)) \binom{n-v_{L_1}}{k-j}$.
\end{proposition}

\begin{proof}
We first observe that an essentially identical proof to that of Proposition~\ref{prop:dfs:manyqueriesfromJ}
shows the analogous result in this case: here we have that $\ell_J \le L_0$ and $\Delta_i(E_t) \le \eps n^{j-i}$ because
the stopping conditions \ref{bfsstop:length} and~\ref{bfsstop:degree} have not been invoked,
and we obtain that the number eligible $k$-sets is at least
$(1-O(\eps))\binom{n-|\Forb| - v_{L_0}}{k-j}$.
It remains only to observe that
$$
(1-O(\eps))\binom{n-|\Forb| - v_{L_0}}{k-j} \ge \binom{n-v_{(1-\delta/3)L_1}}{k-j},
$$
which holds since $|\Forb| \le \delta^2 n$, since $L_1 = \Theta(n)$ and since $\eps \ll \delta \ll 1$.
\end{proof}

\begin{claim}\label{claim:bfsstop:deathtime}
Whp~\ref{bfsstop:death} is not invoked.
\end{claim}

\begin{proof}
In order for~\ref{bfsstop:death} to be invoked, all $j$-sets which became active would need to be fully explored at time $\Tstop$.
Let $m:=\Gdisc(\Tstop) =  E_{\Tstop}$ denote the number of $j$-sets which became active (or were active initially).

By Proposition~\ref{prop:bfs:manyqueriesfromJ}, the algorithm has made at least
$M:=m \cdot (1+\Theta(\delta))\binom{n-v_{L_1}}{k-j}$ queries, from which we certainly discovered at most
$m':=m\binom{k-j}{a}^{-1}$ edges (since each edge gives rise to $\binom{k-j}{a}$ new active $j$-sets).
Thus we have
$$
M/m' = (1+\Theta(\delta)) \binom{n-v_{L_1}}{k-j}\binom{k-j}{a} \ge \frac{1}{(1-\delta^2)p},
$$
or in other words $m' \le (1-\delta^2)pM$.
Thus we have made at least $M$ queries during which we discovered at most $m'\le (1-\delta^2)pM$
edges. A Chernoff bound will show that this is very unlikely provided $pM$ is large enough.

More precisely,
note that since the $j$-sets of $\cJ$ were initially active,
we certainly have $m\ge |\cJ| \ge (\log n)^2$, and therefore $M \ge \Theta((\log n)^2 n^{k-j})$.
For any $t \ge \Theta((\log n)^2 n^{k-j})$, the probability that in the first $t$ queries we find at most
$(1-\delta^2)pt$ edges is at most $\exp(-\Theta(pt)) \le \exp(-\Theta((\log n)^2)) = o(n^{-k})$.
Therefore we may take a union bound over all times $t$ between $0$ and $n^k$ (which is a trivial
upper bound on the total number of queries that can be made) and deduce
that whp \ref{bfsstop:death} was not invoked during this time.
\end{proof}

\begin{proposition}\label{prop:bfsstop:length}
Whp \ref{bfsstop:length} is not invoked.
\end{proposition}

\begin{proof}
We consider the \emph{generations} of the BFS process, where the $j$-sets of
$\cJ$ form generation $0$ and a $j$-set lies in generation $i$ if it was discovered
from a $j$-set in generation $i-1$.
Observe that
since for each $J \in \cJ$ we have $\ell_J \le (1-\delta/2)L_1 + (\log n)^2$,
\ref{bfsstop:length} can only be invoked
if we have reached generation at least $(\log n)^2$.

Let us define $X_i$ to be the number of $j$-sets in generation $i$, so in particular $X_0 = |\cJ| = (\log n)^2$ deterministically.
By Proposition~\ref{prop:bfs:manyqueriesfromJ}, we make at least $X_i (1+\Theta(\delta))(\binom{k-j}{a}p)^{-1}$ queries to obtain generation~$i+1$
from generation~$i$, and therefore $\EE(X_{i+1} | X_i = x_i) \ge (1+\Theta(\delta))x_i$.
A repeated application of the Chernoff bound and a union bound over all at most $n^j$ generations
shows that whp $X_i \ge (1+\delta^2)^i X_0$ for every $i$
until time $\Tstop$ (after which Proposition~\ref{prop:bfs:manyqueriesfromJ} no longer applies).

In order to reach generation $(\log n)^2$, this would involve discovering a generation of size at least
$(1+\delta^2)^{(\log n)^2} X_0 > n^j$, which is clearly impossible since this is larger than the total number of 
$j$-sets available (and indeed~\ref{bfsstop:manyends} would already have been applied long before this point).
\end{proof}

\begin{lemma}\label{lem:bfsstop:degree}
Whp~\ref{bfsstop:degree} is not invoked.
\end{lemma}

\begin{proof}
The proof is broadly similar to the proof of Lemma~\ref{lem:dfsstop:degree}. 
The case when $i=0$ is trivial, since 
$\Delta_0(t) = |\Gdisc(t)|  \le \eps^2 n^j$ because of~\ref{bfsstop:manyends}.
Let us therefore suppose that $1 \le i \le j-1$ and $I$ is an $i$-set.
We observe that there are two ways in which the degree of $I$ can increase in $\Gdisc$.
\begin{itemize}
\item A \emph{new start} at $I$ consists of a $j$-set of $\cJ$ which contains $I$. This contributes one to the degree of $I$.
\item A \emph{jump} to $I$ occurs when a $k$-set containing $I$ is queried from a $j$-set not containing $I$ and this $k$-set is indeed an edge.
Then the degree of $I$ increases by at most $\binom{k-j}{a}$.
\item A \emph{pivot} at $I$ occurs when an edge is discovered from a $j$-set already containing $I$.
Then the degree of $I$ increases by at most $\binom{k-j}{a}$.
\end{itemize}

We bound the contributions made by new starts, jumps and pivots separately.

\subsubsection*{New starts}

Note that in contrast to the DFS algorithm,
 new starts are already determined by the input, since we start only once with this
input and terminate the algorithm if we have no more active $j$-sets.
It is also clear that since the $j$-sets of $\cJ$ lie within a path, the degree of $I$ in $\cJ$ is $O(1)$
(deterministically).

\subsubsection*{Jumps}

The number of $j$-sets of $\Gdisc$ which intersect $I$ in $z\le i-1$ vertices
is at most $\binom{i}{z} \Delta_z(\Gdisc) = O(\eps c_z n^{j-z})$,
where we have used the condition~\ref{bfsstop:degree} with~$z$ in place of~$i$.

Furthermore, each such $j$-set gives rise to $O(n^{k-j-i+z})$ queries which would result in a jump to $I$
if the corresponding $k$-set is an edge, and thus the total number of queries which would result
in a jump to $I$ is
$\sum_{z=0}^{i-1} O(\eps c_z n^{j-z})\cdot O(n^{k-j-i+z}) = O(\eps c_{i-1} n^{k-i})$.
Each such query gives a jump with probability $p=O(n^{k-j})$, and a Chernoff bound
implies that with probability at least $1-\exp(-\sqrt{n})$ the number of jumps is $O(\eps c_{i-1}n^{j-i})$.
Since each jump gives a contribution of at most $\binom{k-j}{a}=O(1)$ to the degree of $I$ in $\Gdisc$,
the total contribution is $O(\eps c_{i-1}n^{j-i})$.

\subsubsection*{Pivots}

For each $j$-set $J$ arising from either a new start
at $I$ or a jump to $I$, we start a new \emph{pivot process} consisting of
all of the $k$-sets we discover from $J$ and its descendants which contain $I$.
It is important that, while we now have a BFS rather than
a DFS process, we are still constructing paths (albeit many simultaneously)
and therefore the number of consecutive pivots at $I$ is at most $\frac{k-i}{k-j} \le k$.
Therefore each pivot process runs for at most $k$ generations.

Furthermore, the number of queries made from each $j$ set in the pivot process
is at most $n^{k-j}$, and therefore the expected number of pivots discovered
from each $j$-set is at most $pn^{k-j}=O(1)$. Since each pivot gives rise to at most
$\binom{k}{j}=O(1)$ many $j$-sets, the expected total size of each pivot process is $O(1)$.

Now since we start at most $O(\eps c_{i-1}n^{j-i})$ pivot processes,
and the expected size of each of these is bounded, it is easy to see that with very high
probability, the total size of all of these pivot processes is $O(\eps c_{i-1} n^{j-i})$.
Indeed, this can be shown using a Chernoff bound on the total number of edges we discover
in all the pivot processes combined, and using the fact that $\eps c_{i-1} n^{j-i} \ge \eps n\ge \sqrt{n}$.
Furthermore the failure probability provided by the Chernoff bound is at most $\exp(-\sqrt{n})$.

\vspace{0.3cm}

Summing the contributions from new starts, jumps and pivots, and applying a union bound
on the failure probabilities, we deduce that with probability at least $1-2\exp(-\sqrt{n})$
the degree of $I$ in $\Gdisc$ is at most
$O(1) + O(\eps c_{i-1}n^{j-i}) \le \eps c_i n^{j-i}$.
A union bound over all $O(n^i)=o(\exp(\sqrt{n}))$ choices of $I$ completes the argument.
\end{proof}

Now combining Claim~\ref{claim:bfsstop:deathtime}, Proposition~\ref{prop:bfsstop:length}
and Lemma~\ref{lem:bfsstop:degree} immediately gives the following corollary.

\begin{corollary}\label{cor:bfsstop:manyends}
Whp stopping condition~\ref{bfsstop:manyends} is invoked first.\qed
\end{corollary}

\subsection{Proof of Lemma~\ref{lem:manyends}}

We can now use this corollary to prove Lemma~\ref{lem:manyends}.

As described above, we let $P_0'$ be a path of length $(1-\delta/2)L_1 + 2(\log n)^2$ in $\Hknp$,
which exists whp by Lemma~\ref{lem:longpath}, and
let $P_0$ be the path obtained by removing $(\log n)^2$ edges from the start and from the end of $P_0'$.
Let $\cJ_1,\cJ_2$ be the sets of $j$-sets which can form the end of a path which lies within $P_0'$ but not within $P_0$,
divided into two classes in the natural way.

We first run the \frayalg\ algorithm with input $k,j,H=\Hknp,P_0,P_0',\cJ=\cJ_1$
and with forbidden vertex set $\Forb=\emptyset$.
Let $\cA$ be the resulting outcome of $\Gdisc(t)=A_t\cup E_t$ at the stopping time $t=\Tstop$.

We now aim to run the algorithm again, with $\cJ_2$ in place of $\cJ_1$.
However, it is in theory possible that all of the augmenting paths ending in a $j$-set of $\cA$ share
some common vertex $x$, and that the same happens when we construct augmenting paths at the other end,
meaning that all pairs of paths will be incompatible.
Of course, intuitively this is very unlikely. To formalise this intuition, we will make use of
our ability to forbid a set $\Forb$ of vertices, which we did not need to do in the first iteration.

Let us define a \emph{heavy} vertex to be a vertex which does not lie in $P_0'$, but lies in at least $\eps^2 (\log n)^3 n^{j-1}$ many
augmented paths $P_J$, where $J \in \cA$ (i.e.\ which lies in at least a $(\log n)^3/n$ proportion of the augmented paths).

\begin{claim}
Whp there are at most $\delta^2 n$ heavy vertices.
\end{claim}

\begin{proof}
Let $q$ be the number of pairs $(v,P)$ consisting of a heavy vertex $v$
and an augmenting path containing $v$. We will estimate $q$ in two different ways.

By~\ref{bfsstop:manyends}, there are at most $\eps^2 n^j$ choices for $P$,
each of which contains at most $j+2(\log n)^2(k-j) \le 2k(\log n)^2$ vertices due to~\ref{bfsstop:length}, and therefore
certainly at most $2k(\log n)^2$ heavy vertices $v$, which implies that $q \le \eps^2 n^j \cdot 2k(\log n)^2$.

On the other hand, letting $h$ denote the number of heavy vertices,
we have $q \ge h \cdot \eps^2 (\log n)^3 n^{j-1}$ by the definition of a heavy vertex.

Combining these two estimates, we obtain $h \le 2kn/(\log n) \le \delta^{2}n$.
\end{proof}

We now run the algorithm again, this time with input $\cJ=\cJ_2$ and with $\Forb$ being
precisely the set of heavy vertices (and all other inputs as before).
Let $\cB$ be the resulting outcome of $\Gdisc = A_t \cup E_t$ at the stopping time.
We claim that whp $P_0$, $\cA$ and $\cB$ satisfy the conditions of Lemma~\ref{lem:manyends}.

First, recall that $P_0$ has length $(1-\delta/2)L_1$ by definition.
Next, observe that by Corollary~\ref{cor:bfsstop:manyends} we have $|\cA|,|\cB|= \eps^2 n^j$ whp.

Let $J_s,J_e$ be the two end $j$-sets of $P_0$.
It is clearly true by construction that every $j$-set $J\in A$ augments $(P_0,J_s)$
while every $j$-set $J\in B$ augments $(P_0,J_e)$ (without loss of generality this way round).
Furthermore, the length of the augmenting paths
is at most $2(\log n)^2$
due to~\ref{bfsstop:length}.

Finally, we show the trickiest of the properties, that for most pairs $(A,B) \in \cA \times \cB$
the augmenting paths are disjoint. Given $A \in \cA$, let $P_{A,J_s}$ denote the corresponding augmenting path
(i.e.\ $P_A$ without $P_0-J_s$),
and let $\mathscr{P}_\cA = \{P_{A,J_s}: A \in \cA\}$ denote the set of these augmenting paths.
For $B \in \cB$, we define $P_{B,J_e}$ and $\mathscr{P}_\cB$ similarly.
Observe that for each $B \in \cB$ the path $P_{B,J_e}$ has length at most $2(\log n)^2$, 
and therefore contains $O((\log n)^2)$ vertices.
Since when constructing $\cB$ we excluded heavy vertices, each of these vertices lies
in at most $\eps^2(\log n)^3 n^{j-1}$ of the paths $\mathscr{P}_\cA$,
and therefore each path of $\mathscr{P}_\cB$ intersects with at most $O\left(\eps^2 (\log n)^5 n^{j-1}\right) = O\left(\frac{(\log n)^5}{n}|\cA|\right)$
of the paths in $\mathscr{P}_\cA$.
Therefore the number of pairs $(A,B) \in \cA \times \cB$ such that the paths $P_{A,J_s},P_{B,J_e}$ are vertex-disjoint
is at least
$$
|\cB|\left(1-O\left(\frac{(\log n)^5}{n}\right)\right)|\cA| \ge (1-\eps)\eps^4 n^{2j},
$$
as required.

\bibliographystyle{amsplain}
\bibliography{../References}

\end{document}